\documentclass[11pt]{article}%
\usepackage{amssymb, amsmath, latexsym, mathrsfs, verbatim, calc}
\usepackage[latin1]{inputenc}
\usepackage{amsmath}
\usepackage{amsfonts}
\usepackage{amssymb}
\usepackage{graphicx}
\usepackage{color}%
\providecommand{\U}[1]{\protect\rule{.1in}{.1in}}

\def\CV{\operatorname{\rm CV}}

\newtheorem{theorem}{Theorem}

\newtheorem{corollary}[theorem]{Corollary}

\newtheorem{definition}[theorem]{Definition}

\newtheorem{lemma}[theorem]{Lemma}

\newtheorem{proposition}[theorem]{Proposition}
\newtheorem{remark}[theorem]{Remark}

\newenvironment{proof}[1][Proof]{\noindent\textbf{#1.} }{\ \rule{0.5em}{0.5em}}

\newcommand{\N}{\mathbb{N}}

\newcommand{\R}{\mathbb{R}}

\def\N{\mathbb{N}}
\def\R{\mathbb{R}}
\def\P{\mathbb{P}}
\def\E{\mathbb{E}}
\def\U{\mathbb{U}}

\begin{document}

\title{ Penalization for a PDE with a Nonlinear Neumann boundary condition and  measurable coefficients  
	 \thanks{This work is supported by PHC Toubkal 18/59} }
\author{Khaled Bahlali\thanks{Université de Toulon, IMATH, EA $2134$, $83957$ La Garde
cedex, France.}
\and Brahim Boufoussi \thanks{LIBMA, Department of Mathematics, Faculty of Sciences Semlalia, Cadi Ayyad University, 2390 Marrakesh, Morocco}
\and Soufiane Mouchtabih \footnotemark[2] \footnotemark[3] }
\date{}
\maketitle

\footnotetext{{E-mail addresses: khaled.bahlali@univ-tln.fr (Khaled Bahlali), boufoussi@uca.ma (Brahim Boufoussi), soufiane.mouchtabih@gmail.com (Soufiane Mouchtabih)}}
\maketitle



\noindent\textbf{Abstract}: We consider a system of semi-linear partial differential equations with measurable coefficients and a nonlinear Neumann boundary condition. We 
then construct a sequence of penalized partial differential equations which converges to a solution of our initial problem.   The solution we construct is in the $L^p-$viscosity sense, since the coefficients can be not continuous. 
The method we use is based on  backward stochastic differential equations and their  S-tightness. 
The present work is motivated by the fact that many partial differential equations arising in physics have discontinuous coefficients.   

\section{Introduction}

Let $D$ be a $\mathcal{C}^2$ convex, open and bounded domain in $\mathbb{R}^d$, and for $(t,x)\in [0,T]\times\bar{D}$ we consider the following reflecting stochastic differential equation
\begin{equation*}\label{refsde}
X_s=x+\int_t^sb(X_r)dr+\int_t^s\sigma(X_r)dW_r+K_s,\quad s\in[t,T],
\end{equation*}
where $b :\mathbb{R}^d\to \mathbb{R}^d$, $\sigma:\mathbb{R}^d\to \mathbb{R}^{d\times d'}$ are given measurable functions and $K$ is a bounded variation process satisfying some minimality conditions.
Several authors have studied approximations of reflected
diffusions in such domains. We refer for example to \cite{Me} and \cite{tanaka}
in the case of a convex bounded domain $ D $ and with coefficients satisfying Lipschitz conditions.
The non-convex case was treated in \cite{LiSz} then  
extended to reflected diffusions on non necessary 
bounded domains in \cite{PW}.
A general situation of non Lipschitz coefficients and non convex domain can be found in \cite{RS1}, where the authors studied, in particular, the existence of a weak solution of the reflected equation, when  the coefficients are merely measurable and the diffusion coefficient may degenerate on some subset of the domain. Note that equation $\eqref{refsde}$ can be used to handle linear PDEs with Neumann Boundary conditions, see for instance \cite{Hsu, SatoUeno, StrVara}. 

Our aim in the present work is to construct the solution of a system of semi-linear partial differential equations (PDEs), with a nonlinear Neumann boundary condition by penalization. For this purpose, we use the backward stochastic differential equations.  This allows us to 
 provide probabilistic representations for solutions of different type of semilinear PDEs, see  for instance \cite{PardPen} for parabolic equations, \cite{DarPard} for elliptic equations with Dirichlet boundary condition and \cite{PardZhang}  for a nonlinear Neumann boundary condition. More references can be found in \cite{PardRas}. 
 
 The penalization of nonlinear Neumann boundary problem \eqref{limitsytem} has been firstly considered in \cite{BouCas} when the coefficients  $b$, $\sigma$ are uniformly Lipschitz then extended by \cite{BahMatZal} to the case where the coefficients $b$, $\sigma$ are continuous. The main goal of the present paper is to extend the results of  \cite{BahMatZal}, \cite{BouCas} to the situation where the coefficients $b$ and $\sigma$ are merely measurable and the nonlinearity $f$ is measurable in $x$.. Our work is motivated by the fact that in many problems arising in physics. Our method is inspired from  that developed in \cite{BahMatZal,BouCas}. The difficulty in our situation is due to the discontinuity of the coefficients which makes the convergence of the sequence of penalized equations more delicate. Moreover, due to the non continuity of the coefficients, the classical viscosity solution, which is used in \cite{BahMatZal, BouCas, PardZhang}, can not be defined for our PDEs. We therefore use the notion of $L^p$-viscosity solution introduced in \cite{caf et al} for which we give here a probabilistic interpretation. More details on this topic can be found in \cite{caf et al} and \cite{CranKoc}. 

To describe our result, we shall recall some notations which will be used in the sequel.
 We assume that there exists a function $l\in\mathcal{C}^2_b(\mathbb{R}^d)$ such that
$$D=\{x\in \mathbb{R}^d:l(x)>0\},\quad \partial D=\{x\in \mathbb{R}^d: l(x)=0\},$$
and for all $x\in \partial D$, $\nabla l(x)$ is the unit normal pointing toward the interior
of $ D$. In order to define the approximation procedure we consider the application $x\mapsto dist^2(x,\bar{D})$, therefore, this function is $\mathcal{C}^1$ and convex on $\mathbb{R}^d$. On the other hand we can choose $l$ such that
$$<\nabla l(x),\delta(x)>\le 0,\quad \forall x\in \mathbb{R}^d,$$
where $\delta(x):=\nabla\left(  dist^2(x,\bar{D})\right) $ is called the penalization term. We have
$$\frac{1}{2}\delta(x)=\frac{1}{2}\nabla \left( dist^2(x,\bar{D})\right)=x-\pi_{\bar{D}}(x),\quad \forall x\in\mathbb{R}^d $$
where $\pi_{\bar{D}}$ is the projection operator. Moreover, $\delta$ is a Lipschitz function and we have
\begin{equation}\label{projection pro}
	<z-x,\delta(x)>\le 0,\quad \forall x\in \mathbb{R}^d,\quad\forall z\in \bar{D}.
\end{equation}

We consider the following sequence of semi-linear partial
differential equations ($1\leq i\leq k$, $0\leq t\leq T$, $x\in\R^{d}$,
$n\in\N$).
\begin{equation}\label{pensystem}
\left\{
\begin{aligned}
\displaystyle
&\frac{\partial u^{n}_{i}}{\partial t}(t,x) +\mathcal{L}\, u^{n}_{i}(t,x)+
f_{i}(t,x,u^{n}(t,x))\\
&\qquad -n<\delta(x),\nabla
u^{n}_{i}(t,x)>-n<\delta(x),\nabla l(x)
> h_{i}(t, x, u^n(t,x)) = 0\,; \\
&u^{n}(T, x)  =  g(x)\, .\hfill\cr
\end{aligned}\right.
\end{equation}
where $\mathcal{L}$ is the infinitesimal generator corresponding to the diffusion part of $ X $, that is
\begin{equation*}
\mathcal{L}=\frac{1}{2}\sum_{i,j}\left( \sigma\sigma^*(.)\right)_{ij}\frac{\partial^2}{\partial x_i \partial x_j}+\sum_{i}b_i(.)\frac{\partial}{\partial x_i}.
\end{equation*}
 that $(t,x)$ belongs to
$[0,T]\times{\bar D}$.
Under suitable assumptions on the coefficients $ f $, $ g $ and $ h $, by the mean of the connection between BSDEs and semi-linear PDEs, we prove that the sequence $u^{n}(t,x)$ converges, as $ n $ goes to infinity, to a function $u(t,x)$, which is the
solution in the $L^p-$viscosity sense, of the following PDE with Neumann
boundary condition:
\begin{equation}\label{limitsytem}
\left\{
\begin{matrix}
\displaystyle\frac{\partial u_{i}}{\partial t}(t,x) +\mathcal{L}u_{i}(t,x)+
f_{i}(t,x,u(t,x))  = 0\,,\  1\leq i\leq k\,,\,\,(t,x)\in[0,T)\times D\ ,
\hfill\\
\noalign{\medskip}
u_i(T,x)  =  g_i(x)\,,\ \  x\in D\hfill\\
\displaystyle\frac{\partial u_i}{\partial n}(t,x)+h_i(t, x, u(t,x))=0
\,,\,\,\forall
(t,x)\in[0,T)\times\partial D\,.\hfill
\end{matrix}\right.
\end{equation}
where $\frac{\partial u}{\partial n}$ is the outward normal derivative of $u$ on the boundary of the domain and $\frac{\partial u_i}{\partial n}(t,x)=<\nabla l(x),\nabla u_i(t,x)>$ for all $x\in \partial D$. It turns out that, even when  the coefficients are merely measurable,  the convergence of $u^n$ to $u$ follows from the uniqueness in law of the forward part.

Throughout the paper, $\mathcal{C}([0,T],\mathbb{R}^d)$ is the space of $\mathbb{R}^d$-valued continuous function, $\mathcal{D}([0,T],\mathbb{R}^d)$ is the space of $\mathbb{R}^d$-valued cadlag functions and $W^{1,2}_{p,loc}\left( [0,T]\times\mathbb{R}^d\right)$ is the classical Sobolev space of functions $\varphi$ with values in $\mathbb{R}$ such that both $\varphi$  and all the generalized derivatives $\partial_t\varphi$, $\partial_x\varphi$ and $\partial^2_{xx}\varphi$ belong to $L^p_{loc}([0,T]\times\mathbb{R}^d)$. Furthermore,
 for a sequence of processes $(Y^n)_n$, $Y^n\xrightarrow[U]{*} Y$ will denotes the convergence in law with respect to the uniform topology and $Y^n\xrightarrow[S]{*} Y$ is the weak convergence with respect to the $S$-topology. See Appendix for a brief presentation of this topology and \cite{Jak} for more details.

The paper is outlined as follows, in Section 2 we prove the convergence of solutions of our reflected SDE as well as our penalized SDE. The continuity of the solution with respect to the initial data is also established for both penalized and reflected SDEs. In Section 3, the same properties are established for the solutions of the BSDEs parts which is our first main result.  Section 4 gives  the  application  to PDEs with nonlinear Neumann boundary condition which is the second main result of this paper. 

\section{Reflected stochastic differential equations}
Throughout the paper $T$ is a fixed strictly positive number and $d, d' \in \mathbb{N}^*$.
Consider a stochastic differential equation with
reflecting boundary condition of the form

\begin{equation}\label{rsde2}
	\left\{
	\begin{aligned}
		&X^{t,x}_{s}=x+\int_{t}^{s} b(X^{t,x}_{r})\, dr+\int_{t}^{s}\sigma(X^{t,x}_{r})\, dW_{r}+K^{t,x}_{s}, \\
		&K^{t,x}_s=\int_t^s\nabla l(X^{t,x}_r)d|K^{t,x}|_{[t,r]},\\
		&|K^{t,x}|_{[t,s]}=\int_t^s1_{\{X^{t,x}_r\in \partial D\}}d|K^{t,x}|_{[t,r]},
	\end{aligned}
	\right.
\end{equation}
 where $t\in [0,T]$, $s\in [t,T]$ and the notation $|K^{t,x}|_{[t,s]}$ stands for the total variation of $K^{t,x}$ on the interval $[t,s]$, we will denote this continuous increasing process by $k^{t,x}_s$. In particular we have
 \begin{equation}
 k^{t,x}_s=\int_t^s<\nabla l(X^{t,x}_r),dK^{t,x}_r>.
 \end{equation}
We say that $(\Omega,\mathcal{F},\mathbb{P},\{\mathcal{F}_s\},W,X,K)$ is a weak solution
 of (\ref{rsde2}) if $(\Omega,\mathcal{F},\mathbb{P},\{\mathcal{F}_s\})$ is a stochastic basis, $W$ is a $d'-$dimensional Brownian motion with respect to this basis, $X$ is a continuous adapted process and $K$ is a continuous bounded variation process such that $X_s\in \bar{D}$ $\mathbb{P}-$a.s, $\forall s\in [t,T]$ and $(X,K)$ satisfies System $(\ref{rsde2})$.

We suppose the following assumptions
\begin{itemize}
	\item[$(A.1)$] $b:\mathbb{R}^d\to\mathbb{R}^d$ and $\sigma:\mathbb{R}^d\to \mathbb{R}^{d\times d'}$ are measurable bounded functions,
	\item[$(A.2)$] There exists $\alpha>0$ such that for all $x\in \mathbb{R}^d$ $\sigma\sigma^{*}(x)\geq \alpha\,I ,$
	\item[$(A.3)$] The weak uniqueness holds for Equation (\ref{rsde2}).
\end{itemize}

The reflecting diffusions with measurable coefficients were considered in  \cite{RS1} and \cite{Slo2013} where the authors have proved some approximations, stability and existence results. It should be pointed out that in the case of no continuity of coefficients the uniqueness generally failed. Since the weak uniqueness is crucial to prove our main result, we assume that the weak uniqueness holds for Equation $(\ref{rsde2})$ i.e assumption $(A.3)$.

\begin{remark}
	We assume that one of the following sets of assumptions is satisfied
	\begin{enumerate}
		\item $D$ is a semicompact, the dimension $d\le 2$ and assumptions $(A.1)-(A.2)$ hold. 
		\item $b$ is measurable bounded, $\sigma$ is continuous bounded and $\sigma \sigma^*$ is uniformly nondegenerate.
	\end{enumerate}
Then the weak uniqueness holds for equation $(\ref{rsde2})$. 

Indeed, let $\big(\Omega,\mathcal{F},\mathbb{P},\left\{  \mathcal{F}_{t}\right\}
_{t\geq0},W,X,K\big)$ be a weak solution of (\ref{rsde2}) and $f\in
C^{1,2}\left(  \left[  0,T\right]  \times\bar{D}\right)  $. Applying Itô's
formula to $f\left(  s,X_{s}\right)  $:%
\begin{equation}%
	\begin{array}
		[c]{l}%
		f\left(  s,X_{s}\right)  =f\left(  t,x\right)  +{\displaystyle\int_{t}%
			^{s}\Big(}\frac{\partial f}{\partial r}+\mathcal{L}f{\Big)}\left(
		r,X_{r}\right)  dr+{\displaystyle\int_{t}^{s}}\left\langle \nabla_{x}f\left(
		r,X_{r}\right)  ,\nabla\ell\left(  X_{r}\right)  \right\rangle dk_{r}%
		\medskip\\
		\quad\quad\quad\quad\quad+{\displaystyle\int_{t}^{s}}\left\langle \nabla
		_{x}f\left(  r,X_{r}\right)  ,\sigma\left(  X_{r}\right)  dW_{r}\right\rangle
		.
	\end{array}\label{Ito uniqueness}
\end{equation}
Since $\sigma\sigma^{\ast}$ is nondegenerate, we use 
Krylov's inequality for reflecting diffusions (see Theorem 5.1 in \cite{LauSlo}) to get for every  $s\in\left[ t,T\right]  ,$%
\[%
\begin{array}
[c]{l}%
\mathbb{E}\displaystyle{\int_{t}^{s}\Big|\Big(}\frac{\partial f}{\partial
	r}+\mathcal{L}f{\Big)}\left(  r,X_{r}\right)  {\Big|}1_{\left\{
	X_{r}\in\partial D\right\}  }dr\medskip\\
\quad\leq C~\left(  \displaystyle\int_{t}^{s}\int_{D}\det\left(  \sigma
\sigma^{\ast}\right)  ^{-1}{\Big(}\frac{\partial f}{\partial r}+\mathcal{L}%
f{\Big)}^{d+1}1_{\left\{  \partial D\right\}  }~{dsdx}\right)
^{\frac{1}{d+1}}=0~{.}%
\end{array}
\]
Thus, equality (\ref{Ito uniqueness}) becomes%
\[%
\begin{array}
[c]{l}%
f\left(  s,X_{s}\right)  =f\left(  t,x\right)  +{\displaystyle\int_{t}%
	^{s}\Big(}\frac{\partial f}{\partial r}+\mathcal{L}f{\Big)}\left(
r,X_{r}\right)1_{\left\{  X_{r}\in D\right\}  }%
dr+{\displaystyle\int_{t}^{s}}\left\langle \nabla_{x}f\left(  r,X_{r}\right)
,\nabla\ell\left(  X_{r}\right)  \right\rangle dk_{r}\medskip\\
\quad\quad\quad\quad+{\displaystyle\int_{t}^{s}}\left\langle \nabla
_{x}f\left(  r,X_{r}\right)  ,\sigma\left(  X_{r}\right)  dW_{r}\right\rangle
~,\;\mathbb{P}\text{-a.s.}%
\end{array}
\]
Therefore
\[
f\left(  s,X_{s}\right)  -f\left(  t,x\right)  -{\displaystyle\int_{t}%
	^{s}\Big(}\frac{\partial f}{\partial r}+\mathcal{L}f{\Big)}\left(
r,X_{r}\right)1_{\left\{  X_{r}\in D\right\}  }dr
\]
is a $\mathbb{P}$-submartingale whenever $f\in C^{1,2}\left(  \left[
0,T\right]  \times\bar{D}\right)  $ satisfies%
\[
\left\langle \nabla_{x}f\left(  s,x\right)  ,\nabla\ell\left(  x\right)
\right\rangle \geq0,\forall x\in\partial D.
\]
Under the first set of assumptions we deduce from Theorem 3 in \cite{Kry1} that the process $(X_s)_{s\in[t,T]}$ is unique in law. Under the second set of assumptions  we apply Theorem 5.7 in \cite{StrVara} with $\phi=l$, $\gamma:=\nabla\phi$ and
$\rho:=0$ we obtain that the solution to the submartingale problem is unique
for each starting point $\left( t,x\right)  $, therefore our solution process
$\left(  X_{s}\right)  _{s\in\left[  t,T\right]  }$ is unique in law. Moreover, the uniqueness in law of the couple $(X,K)$ follows from Theorem 6 in \cite{El karoui 2}.  	
\end{remark}

We consider the penalized SDEs related to our reflected diffusion $X^{t,x}$
\begin{equation}\label{PenaSDE}
X^{t,x,n}_s=x+\int_t^s\left[b(X^{t,x,n}_r)-n\delta(X^{t,x,n}_r) \right]dr+\int_t^s\sigma(X^{t,x,n}_r)dW_r,\quad s\in [t,T].
\end{equation}
For $n\in \mathbb{N}$ fixed, under assumptions $(A.1)$ and $(A.2)$, we can deduce from Krylov's works, see \cite{Kry2} and the references therein, that there exists a weak solution of Equation (\ref{PenaSDE}). Moreover, Krylov have also established that it is  possible to select a strong Markov weak solution of Equation (\ref{PenaSDE}). In the sequel we shall need to show the continuity of the flow associated to this equation, for this goal we suppose the following assumption
\begin{itemize}
	\item[$(A.4)$] The weak uniqueness holds for Equation (\ref{PenaSDE}).
\end{itemize}
 \begin{remark}
 We note that in the case of low dimension, $d\le 2$, and assumptions $(A.1)-(A.2)$ are in force, the  assumption $(A.4)$ holds true, see $\cite{Kry1}$ and $\cite{Kry 3}$.
 \end{remark}
We set for all $t\in [0,T]$
\begin{eqnarray*}
K^{t,x,n}_s:=\int_t^s-n\delta(X^{t,x,n}_r)dr\quad\mbox{and}\quad
k^{t,x,n}_s:=\int_t^s<\nabla l(X^{t,x,n}_r),dK^{t,x,n}_r>,\quad \forall s\in [t,T].
\end{eqnarray*}
We recall the following classical boundedness result, (see \cite{BahMatZal}), we have
\begin{equation}\label{boundXn}
\sup_{n\ge 0}\,\mathbb{E}\sup_{s\in[t,T]}|X^{t,x,n}_s|^{2q}+\sup_{n\ge 0}\mathbb{E}|K^{t,x,n}|^q_{[t,T]}<+\infty,\quad \forall q\ge 1.
\end{equation}
The next proposition shows a convergence result of the penalized equation (\ref{PenaSDE}).
\begin{proposition}
	Under the assumptions $(A.1)-(A.3)$. We have:
 $$(X^{t,x,n},K^{t,x,n})\xrightarrow[U\times U]{*} (X^{t,x},K^{t,x}).$$
 Moreover $(X^{t,x},K^{t,x})$ satisfies system $(\ref{rsde2})$.
\end{proposition}
\begin{proof}
	By Theorem 2.1 in \cite{Slo2013}, the process $(X^{t,x,n},K^{t,x,n})$  converges to a solution of Equation $(\ref{rsde2})$. Hence the weak uniqueness gives the result.
\end{proof}

\begin{remark}
Under assumptions $(A.1)-(A.3)$, by virtue of Theorems 7 and 10 in $\cite{El karoui 2}$, the solution $X^{t,x}$ of Equation $(\ref{rsde2})$ is a Markov process.
\end{remark}
We extend  the processes $(X^{t,x},K^{t,x})$ and $(X^{t,x,n},K^{t,x,n})$ to $[0,t]$ by denoting
\begin{equation*}
X^{t,x}_s=X^{t,x,n}_s:=x,\quad K^{t,x}_s=K^{t,x,n}_s:=0,\quad \forall\, s\, \in\,[0,t].
\end{equation*}

Now, by using It\^o's formula, the boundedness of $b$, $\sigma$ and $D$, we obtain a priori estimations for the solutions of (\ref{rsde2}) .
\begin{proposition} Under assumption $(A.1)$.
	We have for all $q\ge 1$
	\begin{equation}
	\sup_{n\ge 1}\mathbb{E}\sup_{s\in [0,T]}|X^{t_n,x_n}_s|^{2q}<\infty,\quad \sup_{n\ge 1}\mathbb{E}|K^{t_n,x_n}|_T^q<\infty.
	\end{equation}
\end{proposition}
We have the following continuity result with respect to the initial data for the solution of the penalized equations (\ref{PenaSDE}).
\begin{proposition} Under assumptions $(A.1)$, $(A.2)$ and $(A.4)$.
	The application  $[0,T]\times\mathbb{R}^d \ni (t,x)\to (X^{t,x,n},K^{t,x,n}) $
	is continuous in law.
\end{proposition}
\begin{proof}
	Let $(t_m,x_m)\to (t,x)$, arguing as in Corollary 2 in \cite{AndSlo}, (see also \cite{Kry2}), using the weak uniqueness we find
	\begin{equation*}
	X^{t_m,x_m,n}\xrightarrow[U ]{\ *\ } X^{t,x,n},
	\end{equation*}
	and we deduce that
	\begin{equation*}
	K^{t_m,x_m,n}\xrightarrow[U ]{\ *\ } K^{t,x,n}.
	\end{equation*}
This ends the proof.
\end{proof}

We now state a continuity in law with respect to the initial data for the solution of equation $(\ref{rsde2})$, which is a slight generalization of Lemma 3.8 in \cite{BahMatZal}.
\begin{proposition} We suppose that $(A.1)-(A.3)$ are in force. Then
	the map $[0,T]\times\bar{D} \ni (t,x)\to (X^{t,x},K^{t,x})$
	is continuous in law.
\end{proposition}
\begin{proof}
Let $(t,x)\in [0,T]\times \bar{D}$ be fixed and $(t_n,x_n)\to (t,x)$,  as $n\to +\infty$. We set
$$(X^{t_n,x_n}_s,K^{t_n,x_n}_s)=(X^n_s,K^n_s).$$
We will prove that the family $(X^n,K^n)$ is tight as family of $\mathcal{C}([0,T],\mathbb{R}^d\times\mathbb{R}^d)-$valued random variables. By It\^o's formula applied to $X^n_s-X^n_r$, where $r$ is fixed and $s\ge r$ we deduce:
\begin{eqnarray*}
\mathbb{E}|X^n_s-X^n_r|^8 &\le& C|s-r|^4+C\mathbb{E}\left(\sup_{v\in[r,s]}\left|\int_r^v<X^n_u-X^n_r,\sigma(X^n_u)dW^n_u> \right|  \right)^4 \\
&\le& C|s-r|^4 +C\mathbb{E}\left( \int_r^s|X^n_u-X^n_r|^2|\sigma(X^n_u)|^2du\right)^2\\
&\le& C|s-r|^4+ C|s-r|^2 \le C\max\{|s-r|^4,|s-r|^2\}.
\end{eqnarray*}
Concerning $K^n$, we have:
$$K^n_s-K^n_r=(X^n_s-X^n_r)-\int_r^sb(X^n_u)du-\int_r^s\sigma(X^n_u)dW^n_u$$
Hence,
\begin{eqnarray*}
\mathbb{E}|K^n_s-K^n_r|^8&\le & C\mathbb{E}|X^n_s-X^n_r|^8+C\mathbb{E}\left(\int_r^s|b(X^n_u)|du \right)^8+C\mathbb{E}\left( \sup_{v\in [r,s]}\left| \int_r^v\sigma(X^n_u)dW^n_u\right| \right)^8  \\
&\le & C\max\{|s-r|^8,|s-r|^2\}.
\end{eqnarray*}
Then $(X^{t,x},K^{t,x})$ is tight on $\mathcal{C}([0,T],\mathbb{R}^d\times \mathbb{R}^d)$ with respect to the initial data $(t,x)$. By Prokhorov's theorem, see Chap I in \cite{PardRas},  there exists a subsequence still denoted by $(X^{n},K^{n})$ such that
$$\left( X^{n},K^{n}\right) \xrightarrow[U\times U ]{\ *\ } \left( X,K\right).$$
We will proceed to the identification of the limits $X\stackrel{law}{=}X^{t,x}$ and $K\stackrel{law}{=}K^{t,x}$.
By the Skorohod's theorem, we can choose a probability space $(\hat{\Omega},\hat{\mathcal{F}},\hat{\mathbb{P}})$, $(\hat{X}^n,\hat{K}^n,\hat{W}^n)$ and $(\hat{X},\hat{K},\hat{W})$ defined on this probability space such that
$$(\hat{X}^n,\hat{K}^n,\hat{W}^n)\stackrel{law}{=}(X^n,K^n,W^n),\quad (\hat{X},\hat{K},\hat{W})\stackrel{law}{=}(X,K,W) $$
and $(\hat{X}^n,\hat{K}^n,\hat{W}^n)\to (\hat{X},\hat{K},\hat{W}) $ $\hat{\mathbb{P}}$-a.s, as $n\to \infty$, where $(\hat{W}^n,\mathcal{F}^{\hat{W}^n,\hat{X}^n})$ and $(\hat{W},\mathcal{F}^{\hat{W},\hat{X}})$ are Brownian motions. We now define
\begin{eqnarray}\label{hat V}
\hat{V}^n_s&:=& x+\int_t^sb(\hat{X}^n_r)dr+\int_t^s\sigma(\hat{X}^n_r)d\hat{W}^n_r,\nonumber \\
\hat{V}_s&:=& x+\int_t^sb(\hat{X}_r)dr+\int_t^s\sigma(\hat{X}_r)d\hat{W}_r.
\end{eqnarray}
Since the processes $X^n$ and $X$ have finite  moments (uniformly in $n$) of any order,  $\sigma$ is non degenerate and the coefficients $b$, $\sigma$ are bounded, then using Skorokhod's representation theorem (\cite{Skor} p. 32) and Krylov's estimate, one can show that: 
$$\int_t^sb(\hat{X}^n_r)dr\xrightarrow[ ]{\ Proba\ } \int_t^sb(\hat{X}_r)dr\quad as \quad n\, \to +\infty, $$
$$\int_t^s\sigma(\hat{X}^n_r)d\hat{W}^n_r\xrightarrow[ ]{\ Proba\ } \int_t^s\sigma(\hat{X}_r)d\hat{W}_r, \quad as \quad n\, \to +\infty.$$

Since $b$ and $\sigma$ are bounded we deduce by the Lebesgue dominated theorem that the following convergence holds in $L^q(\hat{\Omega})$ for each $q\ge 1$,
\begin{equation*}
\hat{\mathbb{E}}\sup_{s\in[t,T]}\left|\hat{V}^n_s-\hat{V}_s \right|^q\to 0,\quad as\quad n\to +\infty.
\end{equation*}
We consider
\begin{equation*}
V^n_s:=x+\int_t^sb(X^n_r)dr+\int_t^s\sigma(X^n_r)dW^n_r.
\end{equation*}
Then $X^n_s=V^n_s+K^n_s$, and we remark that
\begin{equation*}
\left( X^n,K^n,W^n,V^n\right) \stackrel{law}{=}\left( \hat{X}^n,\hat{K}^n,\hat{W}^n,\hat{V}^n\right)\quad on\quad \mathcal{C}\left([0,T],\mathbb{R}^d\times\mathbb{R}^d\times\mathbb{R}^{d'}\times\mathbb{R}^d \right)
\end{equation*}
and
\begin{equation*}
\hat{X}^n_s=\hat{V}^n_s+\hat{K}^n_s\quad \hat{\mathbb{P}}-a.s.
\end{equation*}
We pass to the limits we get
\begin{equation*}
\hat{X}_s=\hat{V}_s+\hat{K}_s \quad \hat{\mathbb{P}}-a.s
\end{equation*}
taking into account of $(\ref{hat V})$, it follows that $(\hat{X},\hat{K})$ is a solution of Equation $(\ref{rsde2})$ with initial data $(t,x)$. By the weak uniqueness we have $(\hat{X},\hat{K})=(X^{t,x},K^{t,x})$. Then $(X^n,K^n)$ converges to $(X^{t,x},K^{t,x})$ as $n\to +\infty$. This achieves the proof.
\end{proof}

The next technical lemma is a stochastic version of Helly-Bray theorem, see Proposition 3.4 in \cite{Zal}.
\begin{lemma}\label{Helly-Bray}
	Let $(M^n,\eta^n):(\Omega^n,\mathcal{F}^n,\mathbb{P}^n)\to \mathcal{C}([0,T],\mathbb{R}^d)$ be a sequence of random variables and $(M,\eta)$ such that
	$$(M^n,\eta^{n})\xrightarrow[U\times U ]{*}(M,\eta).$$
	If $(\eta^n)_n$ has bounded variation a.s. and
	$$\sup_{n\ge 1}\mathbb{P}\left(|\eta^n|_{[0,T]}>a\right)\to 0,\quad as \quad a\to \infty $$
	then $\eta$ has a.s bounded variation and
	$$\int_0^T<M^n_r,d\eta^n_r>\xrightarrow[U ]{\ *\    }\int_0^T<M_r,d\eta_r>, \quad as \quad n\to \infty.$$
	
\end{lemma}
We can immediately deduce from the previous lemma the following convergences.
\begin{lemma}
Assume $(A.1)-(A.4)$. Then we have
		$$k^{t,x,n}\xrightarrow[ U]{\ *\ } k^{t,x}\quad \mbox{and}\quad
		 k^{t_n,x_n}\xrightarrow[ U]{\ *\ }  k^{t,x}.$$
\end{lemma}

\begin{proof}
In view of the convergence $(X^{t,x,n},K^{t,x,n})\xrightarrow[U\times U ]{\ *\ }(X^{t,x},K^{t,x})$ and Lemma \ref{Helly-Bray} applied with $(M^n,\eta^n)=(\nabla l(X^{t,x,n}),K^{t,x,n})$, we get $k^{t,x,n}\xrightarrow[U ]{\ *\ } k^{t,x}$.	
 For the second point, by the continuity in law with respect to  the initial data,  $(X^{t_n,x_n},K^{t_n,x_n})\xrightarrow[U\times U ]{\ *\ }(X^{t,x},K^{t,x})$, again by  Lemma 8 applied this time with $(M^n,\eta^n)=(\nabla l(X^{t_n,x_n}),K^{t_n,x_n})$, we obtain  $k^{t_n,x_n}\xrightarrow[U ]{\ *\ } k^{t,x}$.
\end{proof}
\section{Backward stochastic differential equations}
 Consider the functions $f$, $h:\, [0,T]\times\R^{d}\times\R^{k}
\rightarrow\R^{k} $ and  $ g\,:\,\R^{d}\rightarrow\R^{k} $, satisfying the
following assumptions:
\begin{itemize}
	\item[$(A.5)$] There exist positive constants $C_1$, $C_2$, $l_h$ and $\mu_f \in \mathbb{R}$, $\beta<0$ and $ q\geq 1 $ such that $\forall t, s\in[0,T]$,
	$ \forall\, (x, x^{\prime}, y, y^{\prime})
	\in\left(\mathbb{R}^{d}\right)^{2}\times\left(\mathbb{R}^{k}\right)^{2} $ we have
	\begin{itemize}
		\item[(i)]
		$<y'-y,f(t, x, y')-f(t, x, y)>\,\leq \mu_f\, |y'-y|^2$,
		\item[(ii)]$\mid h(t, x', y')-h(s, x, y)|\leq l_h\,
		(|t-s|+|x'- x|+|y'- y|) $,
		\item[(iii)] $<y'-y,h(t, x, y')-h(t, x, y)>\,\leq \beta\, |y'-y|^2$,
		\item[{\rm (iv)}]\
		$ |f(t,x,y)|+|h(t,x,y)|\leq C_{1}\, (1+|y|) $,
		\item[(v)]$ |g(x)|\leq C_{2} (1+|x|^{q}) $.		
	\end{itemize}
 $g$ is continuous and $f$ is measurable with respect to $x$ and  continuous in $(t,y)$.	
\end{itemize}

We assume without loss of generality that the processes $(X_s^{t,x,n},K^{t,x,n}_s)_{s\in [t,T]}$ and $(X_s^{t,x},K^{t,x}_s)_{s\in [t,T]}$ are considered on the canonical space. Consider the following  generalized BSDEs on $[t,T]$
\begin{eqnarray}
Y^{t,x,n}_{s}&=& g(X^{t,x,n}_{T})+
\int_{s}^{T} f(r, X^{t,x,n}_{r}, Y^{t,x,n}_{r})\, dr
-\int_{s}^{T}U^{t,x,n}_{r}\, dM^{X^{t,x,n}}_{r}\nonumber\\
&&+\int_s^Th(r,X^{t,x,n}_r,Y^{t,x,n}_r)dk^{t,x,n}_r \label{bakseq}
\end{eqnarray}
and
 \begin{equation}
 Y^{t,x}_{s}=g(X^{t,x}_{T})+
 \int_{s}^{T} f(r, X^{t,x}_{r}, Y^{t,x}_{r})\, dr
 -\int_{s}^{T}U^{t,x}_{r}\, dM^{X^{t,x}}_{r}+\int_s^Th(r,X^{t,x}_r,Y^{t,x}_r)dk^{t,x}_r,
 \label{baklimit}
 \end{equation}
 where
 \begin{equation}
M_s^{X^{t,x,n}}:=\int_t^s\sigma(X^{t,x,n}_r)dW_r\quad \text{and}\quad M_s^{X^{t,x}}:=\int_t^s\sigma(X^{t,x}_r)dW_r.
 \end{equation}
Under assumption $(A.5)$, there exist $(Y_s^{t,x,n},U^{t,x,n}_s)_{s\in [t,T]}$ and $(Y_s^{t,x},U^{t,x}_s)_{s\in [t,T]}$ unique solutions of equations $(\ref{bakseq})$ and  $(\ref{baklimit})$ respectively (see \cite{PardZhang}).
\begin{remark}
We note that one or other assumption (A.5)(ii) or (A.5)(iii) is sufficient to ensure the existence and uniqueness of the solutions to both BSDEs \eqref{bakseq} and \eqref{baklimit}. Condition (A.5)(iii) will be used to establish some estimates in goal to prove the tightness proprieties and  (A.5)(ii) is necessary for the identification of the limit.
\end{remark}
The next proposition will be used in order  to get  the convergence of the solutions of the sequence of  penalized PDEs.
\begin{proposition}\label{cv Y^n}
	Assume $(A.1)$-$(A.5)$. The following convergence holds
	$$(Y^{t,x,n},M^{t,x,n},H^{t,x,n})\xrightarrow[ S\times S\times S]{\ *\ } (Y^{t,x},M^{t,x},H^{t,x})$$
	where
	\begin{equation*}
		M^{t,x,n}_s:=\int_t^sU^{t,x,n}_rdM^{X^{t,x,n}}_r,\quad H^{t,x,n}_s:=\int_t^sh(r,X^{t,x,n}_r,Y^{t,x,n}_r)dk^{t,x,n}_r,
	\end{equation*}
	\begin{equation}
		M^{t,x}_s:=\int_t^sU^{t,x}_rdM^{X^{t,x}}_r\quad\mbox{and}\quad H^{t,x}_s:=\int_t^sh(r,X^{t,x}_r,Y^{t,x}_r)dk^{t,x}_r.
	\end{equation}
	Moreover, $\displaystyle\lim_{n\to \infty}Y^{t,x,n}_t=Y^{t,x}_t$.
\end{proposition}
\begin{proof}
The solutions satisfy the following estimate
	\begin{eqnarray*}
		\sup_{n\ge 0}\mathbb{E}\sup_{t\le s\le T}|Y^{t,x,n}_s|^2+\sup_{n\ge 0}\mathbb{E}\int_t^T\|U^{t,x,n}_s\sigma(X^{t,x,n}_s)\|^2ds<+\infty
	\end{eqnarray*}	
for the proof see \cite{BouCas}. To show the tightness property with respect to the $S$-topology we compute the conditional variation $CV_T$ defined in (\ref{Conditional var}) in Appendix.  Arguing as in \cite{BouCas}, we can prove that $(Y^{t,x,n},M^{t,x,n},H^{t,x,n})$ is tight with respect to the $S-$topology, so there exists a subsequence still denoted $(Y^{t,x,n},M^{t,x,n},H^{t,x,n})$ and  $(\bar{Y},\bar{M},\bar{H})$ in $(\mathcal{D}([0,T],\mathbb{R}^k))^3$, such that
\begin{equation}\label{cv: Xn,Kn,Yn,Mn,Hn}
(X^{t,x,n},K^{t,x,n},Y^{t,x,n},M^{t,x,n},H^{t,x,n})\xrightarrow[U\times U \times S\times S\times S]{\ *\ }  (X^{t,x},K^{t,x},\bar{Y},\bar{M},\bar{H}).
\end{equation}
Next, we will pass to the limit and show the convergence of each term in BSDE (\ref{bakseq}). Let's start with  $\int_{s}^{T}f(r,X^{t,x,n}_r,Y^{t,x,n}_r)dr$. It should be noted that the function $f$ may be discontinuous in $x$, then the mapping $(x,y)\to \int_{0}^{T}f(r,x(r),y(r))dr$ from $\mathcal{C}([0,T],\mathbb{R}^d)\times \mathcal{D}([0,T],\mathbb{R}^k)$ to $\mathbb{R}^k$ is not necessary continuous. So, to prove the convergence of this term  we proceed as follows: for $R>0$ let $D_R:=\{x\in \mathbb{R}^d, |x|\le R\}$ and $\tau_R^n:=\inf\{r>t, |X_r^{t,x,n}|>R\, or\, |X_r^{t,x}|>R \}\wedge T$, with convention $\inf\{\emptyset\}=\infty$ and let $f_\eta(t,x,y)=\eta^{-d}\varphi(x/\eta)*f(t,x,y)$, where $\varphi$ is an infinitely differentiable function such that $\int \varphi(x)dx=1$.
\begin{eqnarray}\label{CV: f(Xn,Yn)}
\mathbb{E}\left|\int_{s}^{T\wedge \tau_R^n}f(r,X^{t,x,n}_r, Y^{t,x,n}_r)-f(r,X^{t,x,}_r,\bar{Y}_r)dr\right|\le J_1(n,R,\eta)+ J_2(n,R,\eta)+J_3(R,\eta)
\end{eqnarray}
where
\begin{eqnarray*}
&& J_1(n,R,\eta)=\mathbb{E}\left|\int_{s}^{T\wedge \tau_R^n}f(r,X^{t,x,n}_r, Y^{t,x,n}_r)-f_\eta(r,X^{t,x,n}_r,Y^{t,x,n}_r)dr\right|
\\
&&	J_2(n,R,\eta)=\mathbb{E}\left|\int_{s}^{T\wedge \tau_R^n}f_\eta(r,X^{t,x,n}_r,Y^{t,x,n}_r)-f_\eta(r,X^{t,x}_r,\bar{Y}_r)dr\right|\\
&& J_3(R,\eta)=\mathbb{E}\left|\int_{s}^{T\wedge \tau_R^n}f_\eta(r,X^{t,x}_r, \bar{Y}_r)-f(r,X^{t,x}_r,\bar{Y}_r)dr\right|.
\end{eqnarray*}
Since the function $f_\eta$ is continuous with respect to its three arguments, it follows that the maps  $(x,y)\to \int_{0}^{T}f_\eta(r,x(r),y(r))dr$ is continuous, we pass to the limit in $J_2(n,R,\eta)$  as $n\to +\infty$ we deduce that $J_2(n,R,\eta)$  goes to $0$. Now, consider   $J_1(n,R,\eta)$ and let $M>0$
\begin{eqnarray*}
J_1(n,R,\eta)&\le& \mathbb{E}\int_{s}^{T\wedge \tau_R^n}\left|f(r,X^{t,x,n}_r, Y^{t,x,n}_r)-f_\eta(r,X^{t,x,n}_r,Y^{t,x,n}_r)\right|1_{\{|Y^{t,x,n}_r|>M\}}dr\\
&& + \mathbb{E}\int_{s}^{T\wedge \tau_R^n}\left|f(r,X^{t,x,n}_r, Y^{t,x,n}_r)-f_\eta(r,X^{r,x,n}_r,Y^{t,x,n}_r)\right|1_{\{|Y^{t,x,n}_r|\le M\}}dr\\
&\le & C\, \mathbb{E}\int_{s}^{T\wedge \tau_R^n}(1+\left|Y^{t,x,n}_r\right|)1_{\{|Y^{t,x,n}_r|>M\}}dr\\
&& +\mathbb{E}\int_{s}^{T\wedge \tau_R^n}\sup_{\{|y|\le M\}}\left|f(r,X^{t,x,n}_r,y)-f_\eta(r,X^{t,x,n}_r,y)\right|dr\\
&\le & \frac{C}{M^{1/2}} \left( \mathbb{E}\int_{s}^{T}(1+\left|Y^{t,x,n}_r\right|)^2 dr\right)^{\frac{1}{2}} \left( \mathbb{E}\int_{s}^{T}|Y^{t,x,n}_r|dr\right)^{\frac{1}{2}} \\
&& +\mathbb{E}\int_{s}^{T\wedge \tau_R^n}|\zeta_\eta(t,X^{t,x,n}_r)|dr
\end{eqnarray*}
where $\displaystyle\zeta_\eta(r,x)=\sup_{\{|y|\le M\}}\left|f(r,x,y)-f_\eta(r,x,y)\right| $. Thanks to Krylov's inequality, there exists a positive constant $N(T,R,d)$ such that
\begin{eqnarray*}
	J_1(n,R,\eta)&\le& \frac{C}{M^{1/2}} \left( \mathbb{E}\int_{s}^{T\wedge \tau_R^n}(1+\left|Y^{t,x,n}_r\right|)^2 dr\right)^{\frac{1}{2}} \left( \mathbb{E}\int_{s}^{T}|Y^{t,x,n}_r|dr\right)^{\frac{1}{2}}\\
	&&   +N(T,R,d)\|\zeta_\eta\|_{L^{d+2}([0,T]\times D_R)}
\end{eqnarray*}
passing successively to the limit in $\eta\to 0$ and $M\to +\infty$, it follows that $J_1(n,R,\eta)$ tends to zero for all $n\in \mathbb{N}$. Concerning $J_3(R,\eta)$ similar arguments as above prove the convergence of this term to zero as $\eta$ goes to zero, we note that in the prove of the convergence of this term, we will need some integrability on the process $\bar{Y}$ and this is ensured by Lemma A.2 in \cite{LeJay}.

 Since $\tau_R^n$ is increasing to infinity as $R$ tends to infinity, then for $R$ large enough $T\wedge\tau_R^n=T$.
Finally,
$$\lim_{n\to +\infty}\mathbb{E}\left|\int_{s}^{T}f(r,X^{t,x,n}_r, Y^{t,x,n}_r)-f(t,X^{t,x}_r,\bar{Y}_r)dr\right| =0.$$
Concerning the term $\int_{s}^{T}h(r,X^{t,x,n}_r,Y^{t,x,n}_r)dk^{t,x,n}_r$, we use the Lipschitz continuity of $h$, the week convergence of $k^{t,x,n}$ to $k^{t,x}$ with respect to the uniform topology, together with Lemma 3.3 in \cite{BouCas}, we get that there exists a countable set $Q\subset [0,T)$ such that, for any $s\in [0,T]\backslash Q$,
\begin{eqnarray*}
	\bar{Y}_s&=& g(X^{t,x}_T)+\int_s^Tf(r,X^{t,x}_r,\bar{Y}_r)dr-(\bar{M}_T-\bar{M}_s)+\int_s^Th(r,X^{t,x}_r,\bar{Y}_r)dk^{t,x}_r.
\end{eqnarray*}
Since the processes $\bar{Y}$, $\bar{M}$ and $\bar{H}$ are càdlàg, the previous equality holds true for all $s\in[t,T]$.
Moreover, Lemma A.1 in \cite{LeJay}, ensures that the process $\bar{M}$ is a $\mathcal{F}^{X^{t,x},\bar{Y},\bar{M}}-$martingale. We shall now show that $M^{X^{t,x}}$ is a $\mathcal{F}^{X^{t,x},\bar{Y},\bar{M}}-$martingale. Let $\psi_s$ be a bounded continuous mapping form $\mathcal{C}\left([t,s],\mathbb{R}^d \right)\times\mathcal{D}\left([t,s],\mathbb{R}^k \right)^2$, $\varphi\in \mathcal{C}^\infty_b(\mathbb{R}^d)$ and
$$\mathcal{L}=\frac{1}{2}\sum_{i,j}\left(\sigma\sigma^*(.) \right)_{ij}\frac{\partial^2}{\partial x_i \partial x_j}+\sum_ib_i(.)\frac{\partial}{\partial x_i} $$
be the infinitesimal generator of the diffusion  part of the process $X^{t,x}$.
By It\^o's formula we obtain that
$$\varphi(X^{t,x,n}_s)-\varphi(x)-\int_{t}^{s}\mathcal{L}\varphi(X^{t,x,n}_r)dr-\int_{t}^{s}\nabla\varphi(X^{t,x,n}_r)dK^{t,x,n}_r$$
is a $\mathcal{F}^{X^{t,x,n}}$-martingale.
For any $t\le s_1<s_2\le T$ and for each $n\in \mathbb{N}$, we have
\begin{equation*}
	\begin{split}
		\mathbb{E}&\left[\psi_{s_1}\left(X^{t,x,n},Y^{t,x,n},M^{t,x,n} \right)\left(\varphi(X^{t,x,n}_{s_2})-\varphi(X^{x,t,n}_{s_1})-\int_{s_1}^{s_2}\mathcal{L}\varphi(X^{t,x,n}_r)dr\right. \right. \\
		&\left.\left. -\int_{s_1}^{s_2}\nabla\varphi(X^{t,x,n}_r)dK^{t,x,n}_r \right)   \right]=0.
	\end{split}
\end{equation*}
Moreover,
\begin{eqnarray*}
	&&\lim_{n\to \infty}\mathbb{E}\left[\psi_{s_1}\left(X^{t,x,n},Y^{t,x,n},M^{t,x,n} \right)\left(\varphi(X^{t,x,n}_{s_2})-\varphi(X^{t,x,n}_{s_1})-\int_{s_1}^{s_2}\mathcal{L}\varphi(X^{t,x,n}_r)dr \right)   \right]\\
	&&=\mathbb{E}\left[\psi_{s_1}\left(X^{t,x},\bar{Y},\bar{M} \right)\left(\varphi(X^{t,x}_{s_2})-\varphi(X^{t,x}_{s_1})-\int_{s_1}^{s_2}\mathcal{L}\varphi(X^{t,x}_r)dr\right)   \right].
\end{eqnarray*}
In fact, we will only show the convergence of the term $\psi_{s_1}\left(X^{t,x,n},Y^{t,x,n},M^{t,x,n} \right)\int_{s_1}^{s_2}\mathcal{L}\varphi(X^{t,x,n}_r)dr$
\begin{eqnarray*}
&&\mathbb{E}\left|\psi_{s_1}\left(X^{t,x,n},Y^{t,x,n},M^{t,x,n} \right)\int_{s_1}^{s_2}\mathcal{L}\varphi(X^{t,x,n}_r)dr -\psi_{s_1}\left(X^{t,x},\bar{Y},\bar{M} \right)\int_{s_1}^{s_2}\mathcal{L}\varphi(X^{t,x}_r)dr\right|\\
&&\le \mathbb{E}\left| \left[ \psi_{s_1}\left(X^{t,x,n},Y^{t,x,n},M^{t,x,n} \right) -\psi_{s_1}\left(X^{t,x},\bar{Y},\bar{M} \right)\right] \int_{s_1}^{s_2}\mathcal{L}\varphi(X^{t,x,n}_r)dr\right|\\
&& + \mathbb{E}\left|\psi_{s_1}\left(X^{t,x},\bar{Y},\bar{M} \right)\left[ \int_{s_1}^{s_2}\mathcal{L}\varphi(X^{t,x,n}_r)dr -\int_{s_1}^{s_2}\mathcal{L}\varphi(X^{t,x}_r)dr\right] \right|\\
&& = B_1(n)+B_2(n)
\end{eqnarray*}
in view of (\ref{cv: Xn,Kn,Yn,Mn,Hn}), the continuity of $\psi$ and the boundedness of $b$, $\sigma$, $\varphi$, $\frac{\partial \varphi}{\partial x_i}$ and $\frac{\partial^2 \varphi}{\partial x_i \partial x_j}$ we obtain $\displaystyle\lim_{n\to +\infty} B_1(n)=0$. Concerning $B_2(n)$
\begin{eqnarray*}
B_2(n)&=& \mathbb{E}\left|\psi_{s_1}\left(X^{t,x},\bar{Y},\bar{M} \right)\left[ \int_{s_1}^{s_2}\mathcal{L}\varphi(X^{t,x,n}_r)dr -\int_{s_1}^{s_2}\mathcal{L}\varphi(X^{t,x}_r)dr\right] \right|\\
&\le& C \,\mathbb{E} \int_{s_1}^{s_2}\left|\mathcal{L}\varphi(X^{t,x,n}_r) -\mathcal{L}\varphi(X^{t,x}_r)\right|dr\\
&\le & C \sum_i^d \,\mathbb{E} \int_{s_1}^{s_2}\left|b_i(X^{t,x,n}_r)\frac{\partial \varphi }{\partial x_i}(X^{t,x,n}_r)-b_i(X^{t,x}_r)\frac{\partial \varphi }{\partial x_i}(X^{t,x}_r)\right|dr\\
&+& C \sum_{i,j}^d\,\mathbb{E} \int_{s_1}^{s_2}\left|\left(\sigma\sigma^*(X^{t,x,n}_r) \right)_{ij}\frac{\partial^2 \varphi }{\partial x_i x_j}(X^{t,x,n}_r)-\left(\sigma\sigma^*(X^{t,x}_r) \right)_{ij}\frac{\partial^2 \varphi }{\partial x_i x_j}(X^{t,x}_r)\right|dr
\end{eqnarray*}
using the boundedness of $b$, $\sigma$, $\frac{\partial \varphi}{\partial x_i}$ and $\frac{\partial^2 \varphi}{\partial x_i \partial x_j}$, combined with Krylov's estimate, we proceed as in (\ref{CV: f(Xn,Yn)}) to conclude that $\displaystyle\lim_{n\to +\infty}B_2(n)=0$.
On the other side, using (\ref{cv: Xn,Kn,Yn,Mn,Hn}), the boundedness of $\psi$, $\nabla\varphi$ and estimation in (\ref{boundXn}) together with Lemma \ref{Helly-Bray} we obtain
\begin{eqnarray*}
	&& \lim_{n\to \infty}\mathbb{E}\left[\psi_{s_1}\left(X^{t,x,n},Y^{t,x,n},M^{t,x,n} \right)\int_{s_1}^{s_2}\nabla\varphi(X^{t,x,n}_r)dK^{t,x,n}_r \right]\\
	&&=\mathbb{E}\left[\psi_{s_1}\left(X^{t,x},\bar{Y},\bar{M} \right)\int_{s_1}^{s_2}\nabla\varphi(X^{t,x}_r)dK^{t,x}_r \right].
\end{eqnarray*}
Hence,
\begin{eqnarray*}
	\mathbb{E}\left[\psi_{s_1}\left(X^{t,x},\bar{Y},\bar{M} \right)\left(\varphi(X^{t,x}_{s_2})-\varphi(X^{x,t}_{s_1})-\int_{s_1}^{s_2}\mathcal{L}\varphi(X^{t,x}_r)dr
	-\int_{s_1}^{s_2}\nabla\varphi(X^{t,x}_r)dK^{t,x}_r \right)   \right]=0
\end{eqnarray*}
It\^o's formula gives rise
\begin{equation*}
	\mathbb{E}\left[\psi_{s_1}\left(X^{t,x},\bar{Y},\bar{M} \right)\int_{s_1}^{s_2}\nabla\varphi(X^{t,x}_r)dM^{X^{t,x}}_r \right]=0
\end{equation*}
then, $M^{X^{t,x}}$ is a $\mathcal{F}^{X^{t,x},\bar{Y},\bar{M}}-$ martingale. Since $Y^{t,x}$ and $U^{t,x}$ are $\mathcal{F}^{X^{t,x}}-$adapted, $M^{t,x}=\int_{t}^{.}U^{t,x}_rdM^{X^{t,x}}_r$ is also $\mathcal{F}^{X^{t,x},\bar{Y},\bar{M}}-$ martingale. Therefore, using It\^o's formula, assumptions on $f$ and $h$, and a generalized Gronwall lemma (see Lemma 12 in \cite{MatRas}), we obtain
$$Y^{t,x}=\bar{Y}\quad \mbox{and}\quad M^{t,x}=\bar{M}.$$
For the second claim, By Lemma 3.3 in \cite{BouCas} applied with time $T$, we have
$$\int_t^Th(r,X^{t,x,n}_r,Y^{t,x,n}_r)dk^{t,x,n}_r\xrightarrow[ ]{\ *\ } \int_t^Th(r,X^{t,x}_r,Y^{t,x}_r)dk^{t,x}_r .$$
Since $M^{t,x,n}\xrightarrow[ S]{\ *\ }  M^{t,x}$, using Remark 2.4 in \cite{Jak}, we get $M^{t,x,n}_T\to M^{t,x}_T$ in law.
We now pass to the limit in
\begin{eqnarray*}
	Y^{t,x,n}_{t}&=& g(X^{t,x,n}_T)+\int_t^Tf(r,X^{t,x,n}_r,Y^{t,x,n}_r)dr-M^{t,x,n}_T\\
	&&+\int_t^Th(r,X^{t,x,n}_r,Y^{t,x,n}_r)dk^{t,x,n}_r,
\end{eqnarray*}	
 we deduce that
\begin{eqnarray*}
	Y^{t,x}_{t}&=& g(X^{t,x}_T)+\int_t^Tf(r,X^{t,x}_r,Y^{t,x}_r)dr-M^{t,x}_T\\
	&&+\int_t^Th(r,X^{t,x}_r,Y^{t,x}_r)dk^{t,x}_r.
\end{eqnarray*}
Which ends the proof
\end{proof}

We extend $(Y^{t,x},U^{t,x})$ and $M^{X^{t,x}}$ to $[0,t)$ as follows
$$Y^{t,x}_s:=Y^{t,x}_t,\quad U^{t,x}_s:=0\quad \text{and}\quad M^{X^{t,x}}_s:=0,\quad \forall\, s\, \in\, [0,t).$$
We now state a continuity property of the mappings $(t,x)\to Y^{t,x}$.
\begin{proposition}\label{limYmn}
Assume $(A.1)-(A.3)$ and $(A.5)$. For a sequence $(t_n,x_n)$ converging to $(t,x)$, there exists a subsequence $(t_{n_k},x_{n_k})_{k\in \mathbb{N}}$ such that $Y^{t_{n_k},x_{n_k}}\xrightarrow[S ]{\ *\ } Y^{t,x}$.
\end{proposition}
\begin{proof}
	We denote $(Y^{t_n,x_n},X^{t_n,x_n},k^{t_n,x_n})=(Y^{n},X^{n},k^{n})$.
 We have for all $t\in[0,T]$
	\begin{eqnarray*}
	&&\mathbb{E}\sup_{0\le s\le T}|Y^{n}_s|^2+\mathbb{E}\int_0^T\|U^{n}_s\sigma(X^{n}_s)\|^2ds<C,\quad \forall n \in \mathbb{N},\\
	&&\mathbb{E}\sup_{0\le s\le T}|Y_s|^2+\mathbb{E}\int_0^T\|U_s\sigma(X_s)\|^2ds<C,
	\end{eqnarray*}
where $C$ is a constant independent of and $n$, see \cite{PardZhang}. We compute the conditional variation defined by $(\ref{Conditional var})$ in Appendix, we get
\begin{equation*}
\sup_{n\ge 0}\left(\CV_T(Y^{n})+\mathbb{E}\sup_{0\le s\le T}|Y^{n}_s|+\mathbb{E}\sup_{0\le s\le T}|M^{n}_s|+\CV_T(H^{n})+\mathbb{E}\sup_{0\le s\le T}|H^{n}_s| \right) <\infty.
\end{equation*}
Then, $(Y^{n},M^{n},H^{n})$ is tight with respect to the $S-$topology. So there exists a subsequence still denoted by $(Y^{n},M^{n},H^{n})$ and  $(\bar{Y},\bar{M},\bar{H})$ in $(\mathcal{D}([0,T],\mathbb{R}^k))^3$, such that
$$(X^{t_n,x_n},K^{t_n,x_n},Y^{n},M^{n},H^{n})\xrightarrow[U\times U \times S\times S\times S]{\ *\ }  (X^{t,x},K^{t,x},\bar{Y},\bar{M},\bar{H}).$$
The same arguments used in the proof of Proposition \ref{cv Y^n} ensure that for all $s\in [t,T]$
\begin{eqnarray*}
\bar{Y}_s&=& g(X^{t,x}_T)+\int_s^T1_{[t,T]}f(r,X^{t,x}_r,\bar{Y}_r)dr-(\bar{M}_T-\bar{M}_s)+\int_s^Th(r,X^{t,x}_r,\bar{Y}_r)dk^{t,x}_r
\end{eqnarray*}
and
$$Y^{t,x}=\bar{Y}\quad \mbox{and}\quad M^{t,x}=\bar{M}.$$
Which ends the proof
\end{proof}

The next result will be employed in the sequel.
\begin{proposition}\label{def u^n and u} Under assumptions $(A.1)$-$(A.5)$. The functions $u^n:(t,x)\in [0,T]\times \mathbb{R}^d\mapsto u^{n}(t,x):=Y^{t,x,n}_t$ and $u: (t,x)\in [0,T]\times \bar{D}\mapsto u(t,x):=Y^{t,x}_t$ are continuous.	
\end{proposition}
\begin{proof}
	We will show only that the function $u$ is continuous. Let $(t_n,x_n)\to (t,x)$, as $n\to +\infty$. From the proof of Proposition \ref{limYmn}, we can extract a subsequence still denoted $(t_n,x_n)$, such that
	$$(X^{t_n,x_n},K^{t_n,x_n},Y^{t_n,x_n},M^{t_n,x_n})\xrightarrow[ U\times U \times S \times S]{\ *\ } (X^{t,x},K^{t,x},Y^{t,x},M^{t,x})$$
By Lemma 3.3 in \cite{BouCas} applied for $t=T$, we have
$$\int_0^Th(r,X^{t_n,x_n}_r,Y^{t_n,x_n}_r)dk^{t_n,x_n}_r\xrightarrow[ ]{\ *\ } \int_0^Th(r,X^{t,x}_r,Y^{t,x}_r)dk^{t,x}_r $$
Since $M^{t_n,x_n}\xrightarrow[ S]{\ *\ }  M^{t,x}$, using Remark 2.4 in \cite{Jak}, we get $M^{t_n,x_n}_T\to M^{t,x}_T$. By virtue of Krylov's inequality for reflected diffusions, we can show that $\int_0^T1_{[t_n,T]}f(r,X^{t_n,x_n}_r,Y^{t_n,x_n}_r)dr\to \int_0^T1_{[t,T]}f(r,X^{t,x}_r,Y^{t,x}_r)dr$ in law .
 We now pass to the limit in
\begin{eqnarray*}
u(t_n,x_n)=Y^{t_n,x_n}_{t_n}&=& Y^{t_n,x_n}_0=g(X^{t_n,x_n}_T)+\int_0^T1_{[t_n,T]}f(r,X^{t_n,x_n}_r,Y^{t_n,x_n}_r)dr-M^{t_n,x_n}_T\\
&&+\int_0^Th(r,X^{t_n,x_n}_r,Y^{t_n,x_n}_r)dk^{t_n,x_n}_r.
\end{eqnarray*}	
Exactly as in the proof of the Proposition \ref{limYmn}, we deduce that the limit of $u(t_n,x_n)$, as $n\to +\infty$ is
 \begin{eqnarray*}
 u(t,x)=Y^{t,x}_{t}&=& Y^{t,x}_0=g(X^{t,x}_T)+\int_0^T1_{[t,T]}f(r,X^{t,x}_r,Y^{t,x}_r)dr-M^{t,x}_T\\
 	&&+\int_0^Th(r,X^{t,x}_r,Y^{t,x}_r)dk^{t,x}_r.
 \end{eqnarray*}
 Which is the desired result.
\end{proof}
\section{Application to nonlinear Neumann boundary value problems}
 The coefficients of our PDEs are not continuous this why we cannot define the solutions in the classical viscosity sense. We then adopt the  $L^p-$viscosity solution introduced in \cite{caf et al}. This notion of solutions is used to study  nonlinear PDEs with measurable coefficients. We first recall the definition of the $L^p-$viscosity solution for PDEs $(\ref{pensystem})$. For simplicity, we adopt the following notations
	\begin{eqnarray*}
		\mathcal{L}_n\varphi &:=& \mathcal{L}\varphi -n<\delta(.),\nabla\varphi>\\
		\bar{f}(r,x,y)&:=& f(r,x,y)-n<\nabla l(x),\delta(x)>h(r,x,y).
	\end{eqnarray*}
\begin{definition}
Let $p$ be an integer such that $p>d+2$.
\begin{enumerate}
	\item A function $u\in \mathcal{C}([0,T]\times \mathbb{R}^d, \mathbb{R}^k)$ is a $L^p-$viscosity sub-solution of the PDEs system $(\ref{pensystem})$, if for every $x\in \mathbb{R}^d$, $u_i(T,x)\le g_i(x)$, $1\le i\le k$, and for every $\varphi \in W^{1,2}_{p,loc}([0,T]\times \mathbb{R}^d)$ and $(\hat{t},\hat{x})\in (0,T]\times \mathbb{R}^d$ at which $u_i-\varphi$ has a local maximum, one has
	$$ess\liminf_{(t,x)\to (\hat{t},\hat{x})}\left\lbrace -\frac{\partial \varphi}{\partial t}(t,x)-\mathcal{L}_n\varphi(t,x)-\bar{f}_i(t,x,u(t,x))\right\rbrace\le 0. $$
	\item A function $u\in \mathcal{C}([0,T]\times \mathbb{R}^d, \mathbb{R}^k)$ is a $L^p-$viscosity super-solution of the PDEs $(\ref{pensystem})$, if for every $x\in \mathbb{R}^d$, $u_i(T,x)\ge g_i(x)$, $1\le i\le k$, and for every $\varphi \in W^{1,2}_{p,loc}([0,T]\times \mathbb{R}^d)$ and $(\hat{t},\hat{x})\in (0,T]\times \mathbb{R}^d$ at which $u_i-\varphi$ has a local minimum, one has
	$$ess\limsup_{(t,x)\to (\hat{t},\hat{x})}\left\lbrace -\frac{\partial \varphi}{\partial t}(t,x)-\mathcal{L}_n\varphi(t,x)-\bar{f}_i(t,x,u(t,x))\right\rbrace\ge 0 .$$
	\item A function  $u\in \mathcal{C}([0,T]\times \mathbb{R}^d, \mathbb{R}^k)$ is a $L^p-$viscosity solution if it is both a $L^p-$viscosity sub-solution and super-solution.
\end{enumerate}
\end{definition}
\begin{remark}
	Assertion (1) means that for every $\varepsilon>0$, $r>0$, there exists a set $A\subset B_{r}(\hat{t},\hat{x})$ of positive measure
	$$-\frac{\partial \varphi}{\partial t}(t,x)-\mathcal{L}_n\varphi(t,x)-\bar{f}_i(t,x,u(t,x))\le \varepsilon,\quad \forall (t,x)\in A.$$
\end{remark}
We now define the $L^p-$viscosity solution for system (\ref{limitsytem}), which can be seen as a natural extension of the notion of viscosity solution of PDEs with nonlinear Neumann boundary condition, to the case of PDEs with measurable coefficients.
\begin{definition}\label{DefNeumann}
	 Let $p$ be an integer such that $p>d+2$
	\begin{itemize}
		\item[$(i.)$] $u\in \mathcal{C}([0,T]\times \bar{D},\mathbb{R}^k)$ is called a $L^p-$viscosity subsolution of System $(\ref{limitsytem})$ if $u_i(T,x)\le g_i(x)$, $x\in\bar{D}$, $1\le i\le k$, and moreover for any $1\le i\le k$, $\varphi \in W^{1,2}_{p,loc}([0,T]\times \bar{D})$, and $(\hat{t},\hat{x})\in (0,T]\times\bar{D}$ at which $u_i-\varphi$ has a local maximum, one has
		$$ess\liminf_{(t,x)\to (\hat{t},\hat{x})}\left\lbrace- \frac{\partial \varphi}{\partial t}(t,x)-\mathcal{L}\varphi(t,x)-
		f_{i}(t,x,u(t,x))\right\rbrace \le 0,\quad \text{if}\quad \hat{x}\,\in\, D, $$
		\begin{eqnarray*}		
			\begin{split}
				 &ess\liminf_{(t,x)\to (\hat{t},\hat{x})}\min\left( -\frac{\partial \varphi}{\partial t}(t,x)-\mathcal{L}\varphi(t,x)-
				f_{i}(t,x,u(t,x)), \right. \\
				& \left. -\frac{\partial \varphi}{\partial n}(t,x)-h_i(t, x, u(t,x))  \right)\le 0,\quad \mbox{if}\   \hat{x}\in \partial D.
			\end{split}
		\end{eqnarray*}
		\item[$(ii.)$] $u\in \mathcal{C}([0,T]\times \bar{D},\mathbb{R}^k)$ is called a $L^p-$viscosity super-solution of (\ref{limitsytem}) if $u_i(T,x)\ge g_i(x)$, $x\in\bar{D}$, $1\le i\le k$, and moreover for any $1\le i\le k$, $\varphi \in W^{1,2}_{p,loc}([0,T]\times \bar{D})$, and $(\hat{t},\hat{x})\in (0,T]\times\bar{D}$ at which $u_i-\varphi$ has a local minimum, one has
		$$ess\limsup_{(t,x)\to (\hat{t},\hat{x})}\left\lbrace- \frac{\partial \varphi}{\partial t}(t,x) -\mathcal{L}\varphi(t,x)-
		f_{i}(t,x,u(t,x))\right\rbrace \ge 0,\quad \text{if}\quad \hat{x}\,\in\, D, $$
		\begin{eqnarray*}		
			\begin{split}
				& ess\limsup_{(t,x)\to (\hat{t},\hat{x})}\max\left(- \frac{\partial \varphi}{\partial t}(t,x) -\mathcal{L}\varphi(t,x)-
				f_{i}(t,x,u(t,x)), \right. \\
				& \left. -\frac{\partial \varphi}{\partial n}(t,x)-h_i(t, x, u(t,x))   \right)\ge 0,\quad \mbox{if}\   \hat{x}\in \partial D.
			\end{split}
		\end{eqnarray*}
		\item[$(iii.)$]  $u\in \mathcal{C}([0,T]\times \bar{D},\mathbb{R}^k)$ is called a $L^p-$viscosity solution of System (\ref{limitsytem}) if it is both a $L^p-$viscosity sub- and super-solution.	
	\end{itemize}
\end{definition}
\begin{remark}
	We remark that if the ingredients in the definition above are continuous we recover the classical viscosity solution of PDEs with Neumann boundary condition defined in \cite{PardZhang}.
\end{remark}
We are now able to state and prove our main result.
\begin{theorem}
	Under assumptions $(A.1)-(A.5)$, for $p>d+2$ the functions $u^n: [0,T]\times \mathbb{R}^d\to \mathbb{R}^k$ and $u:[0,T]\times \bar{D}\to \mathbb{R}^k$ are $L^p-$viscosity solutions respectively for systems $(\ref{pensystem})$ and $(\ref{limitsytem})$. Moreover
	$$\lim_{n\to \infty}u^n(t,x)=u(t,x),\quad \forall (t,x)\in [0,T]\times\bar{D},$$
	where $u^n$ and $u$ are defined in Proposition \ref{def u^n and u}.
\end{theorem}
We divide the proof of Theorem 15 in two lemmas and the convergence is ensured by Proposition 10.
\begin{lemma}
	The function $u^n$ is a $L^p-$viscosity solution of system $(\ref{pensystem})$.
\end{lemma}
\begin{proof} The proof will follow the techniques used in Proposition 5.1 in \cite{BahElouPard}. Let $\varphi\in W^{1,2}_{p,loc}\left( [0,T]\times\mathbb{R}^d\right) $, let $(\hat{t},\hat{x})\in [0,T]\times\mathbb{R}^d$ be a point which is a local maximum of $u^n_i-\varphi$. Since $p>d+2$, then $\varphi$ admits a continuous version which we consider from now on. We assume without loss of generality that
\begin{equation}\label{Assm contract}
u^n_i(\hat{t},\hat{x})=\varphi(\hat{t},\hat{x}).
\end{equation}
We will argue by contradiction. Assume that there exists $\varepsilon,\, \alpha\, >0$ such that
\begin{equation}\label{Assm contradiction}
\frac{\partial \varphi}{\partial t}(t,x) +\mathcal{L}_n\, \varphi(t,x)+
\bar{f}_i(t,x,u^{n}(t,x))<-\varepsilon,\quad \lambda-a.e\, \text{in}\, B_{\alpha}(\hat{t},\hat{x})
\end{equation}
where $\lambda$ denote the Lebesgue measure and $B_{\alpha}(\hat{t},\hat{x})$ is the ball of centre $(\hat{t},\hat{x})$ and radius $\alpha$. Since $(\hat{t},\hat{x})$ is a local maximum of $u^n_i-\varphi$, we find a positive number $\alpha'$ (which we can suppose
equal to $\alpha$) such that
	\begin{equation}
	u^n_i(t,x)\le \varphi(t,x)\qquad \text{for all}\, (t,x)\, \text{in}\ B_\alpha(\hat{t},\hat{x}).
	\end{equation}
	Define the stopping time
	\begin{equation*}
	\tau=\inf\{s\ge \hat{t};\ |X^{\hat{t},\hat{x},n}_s-\hat{x}|>\alpha\}\wedge(\hat{t}+\alpha)
	\end{equation*}
	Since $X^{t,x,n}$ is a Markov diffusion, one can show, as in \cite{Elkaroui}, that for every $r\in[\hat{t},\hat{t}+\alpha]$, $Y^{\hat{t},\hat{x},n}_r=u^{n}(r,X^{\hat{t},\hat{x},n}_r)$. Hence, the process $(\bar{Y}_s,\bar{U}_s):=\left(Y^{\hat{t},\hat{x},n,i}_s,\,1_{[0,\tau]}U^{\hat{t},\hat{x},n,i}_s \right)_{s\in [\hat{t},\hat{t}+\alpha]} $ solves the BSDE for every $s\in [\hat{t},\hat{t}+\alpha]$
	\begin{eqnarray}
	 \bar{Y}_s=u^n_i(\tau,X^{\hat{t},\hat{x},n}_\tau)+\int_{s}^{\hat{t}+\alpha}1_{[0,\tau]}\bar{f}_i(r,X^{\hat{t},\hat{x},n}_r,u^{n}(r,X^{\hat{t},\hat{x},n}_r))\, dr-\, \int_{s}^{\hat{t}+\alpha}\bar{U}_r\, dM^{X^{\hat{t},\hat{x},n}}_r.
	\end{eqnarray}
On the other hand, by It\^o-Krylov's formula (see Chap. 2  Sec. 2 and 3 \cite{Kry2}), the process $(\hat{Y}_s,\hat{U}_s)_{s\in [\hat{t},\hat{t}+\alpha]}$ defined by $$(\hat{Y}_s,\hat{U}_s):=\left(\varphi(s\wedge\tau,X^{\hat{t},\hat{x},n}_{s\wedge\tau}),\,1_{[0,\tau]}\nabla\varphi(s,X^{\hat{t},\hat{x},n}_s) \right) $$	
satisfies
\begin{eqnarray*}
\hat{Y}_s&=& \varphi(\tau,X^{\hat{t},\hat{x},n}_\tau)-\, \int_{s}^{\hat{t}+\alpha}1_{[0,\tau]}\left( \frac{\partial \varphi}{\partial r}+\mathcal{L}_n\varphi\right) (r,X^{\hat{t},\hat{x},n}_r)\, dr-\, \int_{s}^{\hat{t}+\alpha}\hat{U}_r dM^{X^{\hat{t},\hat{x},n}}_r.
\end{eqnarray*}
By the choice of $\tau$, $(\tau,X^{\hat{t},\hat{x},n}_\tau)\in B_\alpha(\hat{t},\hat{x})$, then $u^n_i(\tau,X^{\hat{t},\hat{x},n}_\tau)\le \varphi(\tau,X^{\hat{t},\hat{x},n}_\tau)$. \\
Consider the set
$$A:=\{(t,x)\in B_\alpha(\hat{t},\hat{x}),\, \left(\frac{\partial \varphi}{\partial t}+\mathcal{L}_n\varphi+\bar{f}_i(.,.,u^n(.,.)) \right)(t,x)<-\varepsilon \}$$
and
$A^c:= B_\alpha(\hat{t},\hat{x})\backslash A$ is
the complement of A. By assumption $(\ref{Assm contradiction} )$ we get $\lambda(A^c)=0$. Since the process $X^{\hat{t},\hat{x},n}$ is nodegenerate, Krylov's inequality (see Chap. 2  Sec. 2 and 3 \cite{Kry2}) implies that $1_{A^c}(r,X^{\hat{t},\hat{x},n}_r)=0\,\, dr\times\mathbb{P}-a.e$. It follows that
\begin{eqnarray*}
0<\mathbb{E}(\tau -\hat{t})\varepsilon \le \mathbb{E}\int_{\hat{t}}^{\hat{t}+\alpha}-1_{[0,\tau]}\left[\left(\frac{\partial\varphi}{\partial r}+\mathcal{L}_n\varphi \right)(r,X^{\hat{t},\hat{x},n}_r)+\bar{f}_i(r,X^{\hat{t},\hat{x},n}_r,u^n(r,X^{\hat{t},\hat{x},n}_r))  \right]\, dr
\end{eqnarray*}
This implies that
$$-1_{[0,\tau]}\left[\left(\frac{\partial\varphi}{\partial r}+\mathcal{L}_n\varphi \right)(r,X^{\hat{t},\hat{x},n}_r)+\bar{f}_i(r,X^{\hat{t},\hat{x},n}_r,u^n(r,X^{\hat{t},\hat{x},n}_r))  \right]>0$$
on a set of $dt\times d\mathbb{P}-$positive measure. Therefore, the comparison theorem in Remark 2.5 in \cite{Pardoux} shows that $\hat{Y}_{\hat{t}}>\bar{Y}_{\hat{t}}$, that is $\varphi(\hat{t},\hat{x})>u^n_i(\hat{t},\hat{x})$, which contradicts assumption $(\ref{Assm contract})$.
\end{proof}
\begin{lemma}
	The function $u$ is a $L^p-$viscosity solution of system $(\ref{limitsytem})$ in the sense of Definition \ref{DefNeumann}.
\end{lemma}
\begin{proof}
We shall prove that $u$ is a $L^p-$viscosity subsolution. Let $\varphi\in W^{1,2}_{p,loc}([0,T]\times \bar{D})$ and let $(\hat{t},\hat{x})\in [0,T]\times \bar{D}$ be a point which is a local maximum of $u_i-\varphi$. We consider a continuous version of $\varphi$ and we assume without loss of generality that
\begin{equation}\label{assum}
u_i(\hat{t},\hat{x})=\varphi(\hat{t},\hat{x}).
\end{equation}
We skip the proof in the case $\hat{x}\in  D$ because of its similitude with that
of $u^n$ in the previous lemma.
We consider the case $\hat{x}\in \partial D$, we suppose that
\begin{eqnarray*}	
\begin{split}
	 & ess\liminf_{(t,x)\to (\hat{t},\hat{x})} \min\left( -\frac{\partial \varphi}{\partial t}(t,x) -\mathcal{L}\varphi(t,x)-
	 f_{i}(t,x,u(t,x)), \right. \\
	& \left.  -\frac{\partial \varphi}{\partial n}(t,x)-h_i(t, x, u(t,x))  \right)> 0.
\end{split}
\end{eqnarray*}
It follows that there exist $\varepsilon,\alpha>0$ such that
\begin{eqnarray}
	&&\frac{\partial \varphi}{\partial t}(t,x) +\mathcal{L}\varphi(t,x)+
	f_{i}(t,x,u(t,x))<-\varepsilon\quad \mbox{and}\\
	&&\frac{\partial \varphi}{\partial n}(t,x)+h_i(t, x, u(t,x))<-\varepsilon \nonumber\quad \lambda- a.e\quad  \mbox{in }\,  B_\alpha(\hat{t},\hat{x}).
\end{eqnarray}
Since $(\hat{t},\hat{x})$ is a local maximum of $u_i-\varphi$ we have
$$u_i(t,x)\le \varphi(t,x)\quad \mbox{in}\quad B_\alpha(\hat{t},\hat{x})$$
Define
$$\tau:=\inf\{s\ge \hat{t}: |X^{\hat{t},\hat{x}}_s-\hat{x}|>\alpha\}\wedge(\hat{t}+\alpha)$$
Since $X^{t,x}$ is a Markov process, we have $\forall r\in [\hat{t},\hat{t}+\alpha], Y^{\hat{t},\hat{x}}_r=u(r,X^{\hat{t},\hat{x}}_r)$. Moreover, the process $(\bar{Y}_s,\bar{U}_s):=(Y^{\hat{t},\hat{x},i}_{s\wedge \tau},1_{[0,\tau]}(s)U^{\hat{t},\hat{x},i}_s)$ for $s\in [\hat{t},\hat{t}+\alpha]$ solves the equation
\begin{eqnarray*}
	\bar{Y}_s&=& u_i(\tau,X^{\hat{t},\hat{x}}_\tau)+\int_{s}^{\hat{t}+\alpha}1_{[0,\tau]}f_i(r,X^{\hat{t},\hat{x}}_r,u(r,X^{\hat{t},\hat{x}}_r))dr-\int_{s}^{\hat{t}+\alpha} \bar{U}_rdM^{X^{\hat{t},\hat{x}}}_r\\
	&&+\int_{s}^{\hat{t}+\alpha}h_i(r,X^{\hat{t},\hat{x}}_r,u(r,X^{\hat{t},\hat{x}}_r)) dk^{\hat{t},\hat{x}}_r.
\end{eqnarray*}
On the other hand, by It\^o-Krylov's formula, see for example Corollary 3.6 in \cite{BahlaliMez}, the process $(\hat{Y}_s,\hat{U}_s):=(\varphi(s\wedge\tau,X^{\hat{t},\hat{x}}_{s\wedge \tau}),1_{[0,\tau]}\nabla\varphi(s,X^{\hat{t},\hat{x}}_s))$ solves the following BSDE
\begin{eqnarray*}
	\hat{Y}_s&=& \varphi(\tau,X^{\hat{t},\hat{x}}_\tau)-\int_{s}^{\hat{t}+\alpha}1_{[0,\tau]}(r)\left(\frac{\partial \varphi}{\partial t}+\mathcal{L}\varphi \right)(r,X^{\hat{t},\hat{x}}_r)dr- \int_{s}^{\hat{t}+\alpha} \hat{U}_rdM^{X^{\hat{t},\hat{x}}}_r\\
	&&-\int_{s}^{\hat{t}+\alpha}1_{[0,\tau]}(r)\frac{\partial \varphi}{\partial n}(r,X^{\hat{t},\hat{x}}_r)dk^{\hat{t},\hat{x}}_r.
\end{eqnarray*}
We consider the set
$$A=\{(t,x)\in B_\alpha(\hat{t},\hat{x}): \frac{\partial \varphi}{\partial t}(t,x) +\mathcal{L}\varphi(t,x)+
f_{i}(t,x,u(t,x))<-\varepsilon\}$$
then $\lambda(A^c)=0$, where $A^c$ is the complement set of $A$. By Krylov's inequality (see \cite{BahlaliMez}, \cite{Melnikov} and \cite{AndSlo}) we get $1_{A^c}(r,X^{\hat{t},\hat{x}}_r)=0\quad dr\times d\mathbb{P}-a.e$, it follows that
$$\mathbb{E}\int_{\hat{t}}^{\hat{t}+\alpha}-1_{[0,\tau]}(r)\left[(\frac{\partial\varphi}{\partial t}+\mathcal{L}\varphi)(r,X^{\hat{t},\hat{x}}_r)+f_i(r,X^{\hat{t},\hat{x}}_r,u(r,X^{\hat{t},\hat{x}}_r)) \right] dr\ge \mathbb{E}(\tau-\hat{t})\varepsilon>0.$$
Then
$$-1_{[0,\tau]}(r)(\frac{\partial\varphi}{\partial t}+\mathcal{L}\varphi)(r,X^{\hat{t},\hat{x}}_r)>1_{[0,\tau]}(r)f_i(r,X^{\hat{t},\hat{x}}_r,u(r,X^{\hat{t},\hat{x}}_r))$$
on a set of $dr\times d\mathbb{P}$ positive measure. Furthermore, by Theorem 1.4 in \cite{PardZhang} we get $\hat{Y}_{\hat{t}}>\bar{Y}_{\hat{t}}$, which contradicts our assumption $(\ref{assum})$.
\end{proof}
\section*{Appendix}
The $S$-topology on the space $\mathcal{D}([0,T],\mathbb{R}^d)$ was introduced by Jakubowski \cite{Jak}. It is weaker than the Skorokhod topology but stronger than the Meyer-Zheng one in \cite{MeyZheng}. We recall here some relevant results about the $ S $-topology in the case of real paths but they can be extend easily to the case of  finite dimensional space $\mathbb{R}^d$. We have the following propositions.
\begin{proposition}

		\item[(i.)] $K\subset \mathcal{D}([0,T],\mathbb{R})$ is relatively $S$-compact if and only if
		\begin{equation}\label{S1}
		\sup_{x\in K}\sup_{t\in [0,T]}|x_t|<+\infty
		\end{equation}
		and for all $a,b \in \mathbb{R}$ such that $a<b$
		\begin{equation}\label{S2}
		\sup_{x \in K}N^{a,b}(x)<+\infty
		\end{equation}
		where $N^{a,b}$ is the usual number of up-crossings given levels $a<b$, that is, $N^{a,b}(x)\ge k$ if one can find numbers $0\le t_1<t_2<...<t_{2k-1}<t_{2k}\le T$ such that $x_{t_{2i-1}}<a$ and $x_{t_{2i}}>b$, $i=1,2,...,k$.
		\item[(ii.)] $x^n$ converges to $x$ in the $S$-topology if and only if $(x^n)$ satisfies (\ref{S1}), (\ref{S2}) and for every subsequence $(n_k)$, one can find a further subsequence $(n_{k_l})$ and a countable subset $Q\subset [0,T]$ such that $x^{n_{k_l}}_t\to x_t$, $t\in [0,T]\backslash Q$.
\end{proposition}

\begin{corollary}
If $(x^n)$ is relatively $S$-compact and there exists a countable subset $Q$ such that for every $t\in [0,T]\backslash Q$, $x^n_t\to x_t$, then $(x^n)$ converges to $x$.
\end{corollary}
We now recall that a sequence of processes $(X^n)_n$ converges weakly to $X$ in the $S$-topology, $X^n \xrightarrow[S]{\ *\ } X$, if for every subsequence $(X^{n_k})$, we can find a further subsequence $(X^{n_{k_l}})$ and a stochastic processes $(Y_l)$ and $Y$ defined on $([0,1],\mathcal{B}_{[0,1]}, \lambda)$, such that the laws of $Y_l$ and $X^{n_{k_l}}$ are the same, $l\in \mathbb{N}$, for each $\omega \in [0,1]$ $Y_l(\omega)$ converges to $Y(\omega)$ in the $S$-topology, and for each $\varepsilon>0$, there exists an $S$-compact subset $K_\varepsilon\subset \mathcal{D}([0,T],\mathbb{R})$ such that
$$\lambda\left(\left\lbrace\omega \in [0,1]:Y_l(\omega)\in K_\varepsilon, l=1,2,... \right\rbrace  \right)>1-\varepsilon. $$

\begin{proposition}
	The following two properties are equivalent
	\begin{itemize}
	\item[(i.)] $(X^n)$ is $S$-tight.
	\item[(ii.)] $(X^n)$ is relatively compact with respect to the convergence $"\xrightarrow[S ]{\ *\ }"$
	\end{itemize}
\end{proposition}
\begin{proposition}
If $(X^n)$ is $S$-tight and there exists a countable subset $Q\subset [0,T]$ such that for every $j\in \mathbb{N}$ and every $t_1,t_2,...,t_j\in [0,T]\backslash Q$
$$(X^n_{t_1},X^n_{t_2},...,X^n_{t_j})\xrightarrow[]{* }(X_{t_1},X_{t_2},...,X_{t_j})$$
where $X$ is a process with trajectories in $\mathcal{D}([0,T],\mathbb{R})$. Then $X^n\xrightarrow[S]{\ *\ } X$.	
\end{proposition}
On a probability space $\left(\Omega,\mathcal{F},\P\right)$ with a filtration
${\mathcal F}_{t} $,
let $X$ be an adapted process with paths a.s in
$\mathcal{D}([0,T],\mathbb{R})$. If $X_{t}$ is integrable for all
$t\in[0,T]$, we define the conditional variation of $X$ by
\begin{equation}\label{Conditional var}
\CV_{T}(X)=\sup_{\pi}\displaystyle{
	\sum_{i=1}^{n}\E\left[\left|\E[X_{t_{i+1}}-X_{t_{i}}|
	{\mathcal F}_{t_{i}}]\right|\right]}\,,
\end{equation}
where the supremum is taken over all subdivisions $\pi$
of the interval $[0,T]$.
If $\CV_{T}(X)\! <\!\infty $
then the process $ X $ is called a quasi-martingale. Notice
that for martingales $X$ the quantity $\CV_{T}(X)=0$.
\par\smallskip\noindent
We have the following criterion, for the proof we refer for example to \cite{LeJay} and the references therein.
\begin{theorem}\label{th0}
	Let $ (X^{n})_{n\geq1} $ be a family of stochastic process in
	$ \mathcal{D}([0,T],\mathbb{R}) $. If
	\begin{equation}
	\displaystyle\sup_{n\geq1}\left(\CV_{T}(X^{n})+\E\left[
	\sup_{0\leq s\leq T}|X^{n}_{s}|\right]\right)<\infty\, , \label{c4}
	\end{equation}
	then the sequence $\left(X^{n}\right)_{n\geq1}$ is $S$-tight and
	there exists a subsequence
	$\left(X^{n_{k}}\right)_{k\geq1}$ of $\left(X^{n}\right)_{n\geq1}$,
	a process $X$ belonging to
	$\mathcal{D}([0,T],\mathbb{R})$, and a countable subset $Q\subset [0,T)$
	such that for every $j\geq 1$ and for any finite subset
	$\left\{ t_{1},\ldots, t_{j}\right\}$ of
	$[0,T]\setminus Q $ the following convergence is true:
	$$\left(X_{t_{1}}^{n_{k}},\ldots, X_{t_{j}}^{n_{k}}\right)
	\xrightarrow[]{\ *\ }
	\left(X_{t_{1}},\ldots, X_{t_{j}}\right)\,\,\,\, \text{as}\,\, k\to\infty\,.$$
\end{theorem}
\begin{remark}\label{rem21}
	Note that $T$ is not in the countable subset $Q$. More precisely the
	projection $\pi_{T}\,:\, \mathcal{D}([0,T],\mathbb{R})\rightarrow \R$, which assigns
	to $x$ the value $x(T)$, is continuous with respect to the $S$-topology
	({\rm cfr} Remark 2.4. p.8 in \cite{Jak}).
\end{remark}

\end{document}